\documentclass[12pt]{article}
\usepackage[utf8]{inputenc}
\usepackage[english]{babel}
\usepackage[dvips]{graphicx}
\usepackage{indentfirst}
\usepackage{latexsym}
\usepackage{amsthm}
\usepackage{amsfonts}
\usepackage{amssymb}
\usepackage{amsmath}

\usepackage{euscript,bm}
\usepackage{array}
\usepackage{epsf,wrapfig}

\usepackage{hyperref}
\hypersetup{
    colorlinks=true,
    linkcolor=blue,
    filecolor=magenta,
    urlcolor=cyan,
    citecolor=green,
}

\theoremstyle{plain}
\newtheorem{theorem}{Theorem}

\theoremstyle{definition}

\newtheorem{example}{Example}
\newtheorem{property}{Property}  

\newtheorem{conclusion}{Conclusion}
 
\begin{document}

\noindent UDC 519.17

\title{How Trees on Atoms of Subset Algebras Define Minimal Forests and Their Growth} 
\author{V.\,A. Buslov} 
\begin{center}
{\bf How Trees on Atoms of Subset Algebras Define Minimal Forests and Their Growth} 
\end{center}
\begin{center}
{\large V.\,A. Buslov}
\end{center}

\begin{abstract}
A complete description is given of how minimal trees on atoms of the algebra of subsets $\mathfrak{A}_k$ generated by minimal spanning $k$-component forests of a weighted digraph $V$ determine the form of these forests and how forests grow with increasing number of arcs (that is with a decrease in the number of trees). Precise bounds are established on what can be extracted about the tree structure of the original graph if the minimal trees on the atoms of a single algebra $\mathfrak{A}_k$ are known, and also what minimum spanning forests with fewer components can be constructed based on this, and what exactly additional information is required to determine minimum spanning forests consisting of even fewer components.
\end{abstract} 

This paper is a continuation of works \cite{V6,V7,V9} and also uses the results of \cite{V8,V11}. The definitions and notations correspond to those adopted in them, up to the extension of the fonts used for different types of objects, in order to simplify their identification. All the results used are presented in the form of a list of properties.

\section{Designations and definitions} 

 \subsection{Basic definitions}

The set of vertices of an oriented graph $G$ is denoted by ${\tt V}G$, and the set of its arcs is denoted by ${\tt A}G$. 

The basic object is a weighted directed graph $V$, the vertex set of which is denoted separately for convenience by ${\tt N}$, $|{\tt N}|=N$; the arcs $(i,j)\in{\tt A}V$ are assigned real weights $v_{ij}$. We study spanning subgraphs (subgraphs with vertex set ${\tt N}$) of directed graph $V$, which are incoming (entering) forests. An entering forest (hereinafter simply forest) is an directed graph in which no more than one arc comees from each vertex and there are no contours in it. Connected components of a forest are trees. The only vertex of a tree from which an arc does not emanate is the root. 
The set of roots of a forest $F$ is denoted by ${\tt K}_F$.

By $T^F_i$ we mean the maximal subtree of forest $F$ with root at vertex $i$ relative to inclusion. If $i\in{\tt K}_F$, then $T^F_i$ is a connected component of forest $F$.

Since the work only deals with directed graphs, we use the term graph for them.

For the subgraph $G$ of the graph $V$ and the set ${\tt D}\subseteq {\tt N}$, we define the weights   
\begin{equation*}
\Upsilon^G_{\tt D}=\sum_{\begin{smallmatrix}i\in{\tt D} \\ (i,j)\in {\tt A}G\end{smallmatrix}} v_{ij} \ , \ \ \Upsilon^G=\Upsilon^G_{\tt N}=\sum_{(i,j)\in {\tt A}G} v_{ij} \ . 
\end{equation*}
Note that we specifically define the weight so that it also takes into account arcs whose entries do not belong to the set ${\tt D}$. With this definition, there is a property of weight additivity for any directed graph $H$:
\begin{equation}
\Upsilon^H_{\tt D\cup\tt D'}=\Upsilon^H_{\tt D}+ \Upsilon^H_{\tt D'}, \ {\tt D}\cap{\tt D'}=\emptyset, \ \ {\tt D},{\tt D'}\subset {\tt V}H\ .  
\end{equation}

${\cal F}^k$ is a set of spanning forests consisting of $k$ trees.
The minimum weight of $k$-component forests is denoted by
$$\phi^k=\min_{F\in{\cal
F}^k}\Upsilon^F .$$
If ${\cal F}^k=\emptyset$, we set $\phi^k=\infty$. In particular,  $\phi^0=\infty$.

$F\in\tilde{\cal F}^k$ means that $F\in{\cal F}^k$ and $\Upsilon^F=\phi^k$. We will agree to call such forests minimal.

$H=G|_{\tt D}$ is a subgraph of $G$ induced by the set ${\tt D}$, that is, ${\tt V}H={\tt D}$ and ${\tt A}H$ consists of all arcs of $G$, both ends of which belong to the set ${\tt D}$. $G|_{\tt D}$ is also called the restriction of $G$ to the set ${\tt D}$;

${\cal F}^k|_{\tt D}$ is a set of subgraphs of $k$-component spanning forests induced by the set ${\tt D}$;

$\tilde{\cal F}^k|_{\tt D}$ is a set of subgraphs of $k$-component minimum-weight spanning forests induced by the set ${\tt D}$;

$G^F_{\uparrow{\tt D}}$ is a graph obtained from $G$ by replacing  arcs emanating from the vertices of the set ${\tt D}$ with arcs emanating from the same vertices in  graph $F$;

If there is an arc whose outgoing point belongs to the set ${\tt D}$, but whose incoming point does not, then we say that the arc comes from the set ${\tt D}$. Similarly, if there is an arc whose incoming point belongs to ${\tt D}$, but whose outgoing point does not, then we say that the arc comes into ${\tt D}$. The outgoing neighborhood ${\tt N}^{out}_{\tt D}(F)$ of the set ${\tt D}$ is the set of entries of arcs coming in $F$ from the set ${\tt D}$; the incoming neighborhood is defined similarly ${\tt N}^{in}_{\tt D}(F)$. 

It is clear how the outgoing neighborhood of the set ${\tt D}$ is defined for the set ${\cal F}$ of graphs: ${\tt N}^{out}_{\tt D}({\cal F})=\underset {F\in{\cal F}}{\cup}{\tt N}^{out}_{\tt D}(F)$. The incoming neighborhood for the set of graphs is defined similarly.

For any subset ${\tt D}\subset {\tt N}$, its complement $\overline{\tt D}={\tt N}\setminus {\tt D}$.

A family \footnote{We use the term family along with the term set to avoid the phrase "set of sets".} $\mathfrak B $ of nonempty sets ${\tt B}_i$ is called a partition of ${\tt N}$ if ${\tt B}_i\cap{\tt B}_j=\emptyset$ for $i\neq j$, and ${\tt N}=\underset{i}{\cup}{\tt B}_i$. 

Let $\mathfrak{B}$ be a family of subsets of ${\tt N}$. The family $\mathfrak A$ consisting of all possible complements and intersections of these subsets is called the algebra {\it generated} by the family $\mathfrak{B}$.

A non-empty set ${\tt A}$ of an algebra $\mathfrak{A}$ is called an {\it atom} if for any element ${\tt B}$ of $\mathfrak{A}$ either ${\tt A}\cap{\tt B}=\emptyset$ or ${\tt A}\cap{\tt B}={\tt A}$. 

$\mathfrak{B}_k= \{ {\tt V}T^F_i | i\in{\tt K}_F,\  F\in\tilde{\cal F}^k \}$ is a family of vertex sets of $k$-component spanning minimal forests; $\mathfrak{A}_k$ is the algebra of subsets generated by this family; $\aleph_k$ is the family of atoms of this algebra. An atom ${\tt Z}\in\aleph_k$ is called labeled if it contains the root of at least one forest $F\in\tilde{\cal F}^k$. The family of labeled atoms is denoted by $\aleph_k^\bullet$. The remaining atoms are unlabeled, their family is denoted by $\aleph_k^\circ=\aleph_k\setminus \aleph_k^\bullet$.

\subsection{Trees on subsets}

In \cite{V8}, for any subset ${\tt D}$ of the set of all vertices ${\tt N}$, a set of special tree-like minima is defined: 

\begin{equation}
 \lambda_{\tt D}^{\bullet q}=\min_{T\in {\cal T}^{\bullet q}_{\tt D}}\Upsilon^T, \  \lambda_{\tt D}^\bullet= \min_{q\in{\tt D}} \lambda_{\tt D}^{\bullet q},  
\label{bt}
\end{equation} 
where ${\cal T}^{\bullet q}_{\tt D}$ is a set of trees with a set of vertices ${\tt D}$ and a root at the vertex $q\in {\tt D}$.  

Let us also single out subsets of trees $\tilde{\cal T}^{\bullet q}$ on which the corresponding minima are achieved: $T\in \tilde{\cal T}^{\bullet q}_{\tt D}$ $\Leftrightarrow$ $T\in {\cal T}^{\bullet q}_{\tt D}$ and $\Upsilon^T=\lambda_{\tt D}^{\bullet q}$. Let us also introduce the set ${\cal T}^\bullet_{\tt D}$ of trees with the vertex set ${\tt D}$.

Similarly, we select from it a subset $\tilde{\cal T}^\bullet_{\tt D}$ according to the rule: $T\in \tilde{\cal T}^{\bullet }_{\tt D}$ $\Leftrightarrow$ $T\in {\cal T}^{\bullet }_{\tt D}$ and $\Upsilon^T=\lambda_{\tt D}^{\bullet }$. Naturally, the following is true
\begin{equation}
\lambda_{\tt D}^\bullet=\min_{T\in {\cal T}^{\bullet}_{\tt D}}\Upsilon^T \ , \ \ {\cal T}^\bullet_{\tt D}= 
\underset{q\in{\tt D}}{\cup}{\cal T}^{\bullet q}_{\tt D} \ . 
\label{Taus}
\end{equation}

Now let arcs emanate from all vertices of the set $\tt D$ and they form a special tree. In \cite{V11} the set of trees ${\cal T}^\circ_{\tt D}$ is introduced, according to the rule: $T\in{\cal T}^\circ_{\tt D}$ means that
$|{\tt V}T|=|{\tt D}|+1$ and  $T|_{\tt D}\in{\cal T}^\bullet_{\tt D}$. Thus, in the tree $T$, exactly one arc comes from the set ${\tt D}$ itself, and it comes from the root of the induced tree $T|_{\tt D}$.

Note that for different trees $T$ and $T'$ from ${\cal T}^\circ_{\tt D}$, the sets of their vertices, generally speaking, do not coincide, since the root of the tree can be any vertex from $\overline{\tt D}$.

Let's enter the weight
\begin{equation}
\lambda_{\tt D}^\circ =\min_{T\in {\cal T}_{\tt D}^\circ}\Upsilon^{T}   , 
\label{c}
\end{equation} 
and also the subset of $\tilde{\cal T}^\circ_{\tt D}$ trees on which this minimum is achieved: $T\in\tilde{\cal T}^\circ_{\tt D}$ means that $T\in{\cal T}^\circ_{\tt D}$ and $\Upsilon^T=\lambda^\circ_{\tt D}$.

\section{Outgoing restriction }

Here it will be convenient for us to introduce another type of induced subgraphs. Since the work studies incoming forests, and in them no more than one arc comes out from each vertex, then for the selected subset of vertices ${\tt D}$ it turns out to be significant which arcs come out from the vertices of this set. The entries of these arcs may not belong to the set ${\tt D}$. Most proofs are based on replacing arcs emanating from the vertices of some set ${\tt D}$ in one graph with arcs emanating from the same vertices in another graph (the operation of arc replacement). In this connection, it is convenient to introduce the following type of induced subgraphs. By the notation $G|_{\uparrow{\tt D}}$ we will understand the graph $H$, the set of arcs of which are all arcs emanating in the graph $G$ from the vertices of the set ${\tt D}$ (and only these arcs), and the set of its vertices is the set ${\tt D}$ supplemented by the entries of arcs emanating from ${\tt D}$ in the graph $G$: ${\tt V}H={\tt D}\cup{\tt N}^{out}_{\tt D}(G)$. We will call such a graph an {\it outgoing restriction} of the graph $G$ onto the set ${\tt D}$. In this case, the arcs of the graph $G$ that come out from the vertices of the set ${\tt N}^{out}_{\tt D}(G)$ do not fall into the corresponding outgoing restriction even if their entries belong to the set ${\tt D}\cup{\tt N}^{out}_{\tt D}(G)$. 

Note that $G|_{{\uparrow}\tt D}=G|_{\tt D}$ if and only if ${\tt N}^{out}_{\tt D}(G)=\emptyset$.  

If for two graphs $G$ and $H$ we have $G|_{{\uparrow}\tt D}=H|_{{\uparrow}\tt D} $, then obviously $G|_{\tt D}=H|_{\tt D}$, but not vice versa. 

For any tree $T\in{\cal T}^\circ_{\tt D}$, 
$T|_{\uparrow{\tt D}}=T$ and by definition $T|_{\tt D}$ is a tree belonging to the set ${\cal T}^\bullet_{\tt D}$. 

If for a forest $F$ $F|_{\tt D}$ is a tree, and ${\tt K}_F\cap{\tt D}=\emptyset$ then $F|_{\uparrow{\tt D}}$ is also a tree belonging to the set ${\cal T}^\circ_{\tt D}$, since in any forest no more than one arc emanates from any vertex.  

Note also that for two graphs $G$ and $H$, their outgoing restrictions on the set ${\tt D}$ may differ not only in that the corresponding sets of arcs are different, but also in that the sets of their vertices are different, since ${\tt N}^{out}_{\tt D}(G)$, generally speaking, does not coincide with ${\tt N}^{out}_{\tt D}(H)$. This circumstance leads to the fact that if there are two sets of forests ${\cal F}$ and ${\cal F}'$, then the record 

\begin{equation}
{\cal F}|_{\uparrow{\tt D}} \subseteq{\cal F}'|_{\uparrow{\tt D}} 
\label{ff}
\end{equation}
may have different interpretations. We will understand the expression (\ref{ff}) in the following sense. For any $F\in{\cal F}$ there exists $F'\in{\cal F}'$ such that $F|_{\uparrow{\tt D}}= F'|_{\uparrow{\tt D}}$.   

Note that the main method of proving the statements \cite{V6,V7} was the operation of arc replacement. In this case, if there are two graphs $F$ and $G$, and $H=F_{\uparrow{\tt D}}^G$, then obviously $H|_{\uparrow{\tt D}}= G|_{\uparrow{\tt D}}$.   
Taking this circumstance into account, the introduced definition of outgoing restriction allows us to reformulate some provisions \cite{V7} in a more compact form, and also to give some of the statements a stronger interpretation.

\section{Properties used} 

First of all, we note the following property of the weights of trees in minimal forests.

\begin{property}\cite[Proposition 2]{V8} {\it 
Let $F\in\tilde{\cal F}^k$, $k=1,2,\cdots,N$, and ${\tt D}$ be the set of vertices of any of its trees, and vertex $q\in {\tt D}$ be the root of this tree, then
\begin{equation}
\lambda_{\tt D}^\bullet = \lambda_{\tt D}^{\bullet q} =\Upsilon^F_{\tt D}  . 
\label{mlu}
\end{equation}
}
\end{property}
The sets $\tilde{\cal T}_{\tt D}^\circ$ and $\tilde{\cal T}_{\tt D}^{\bullet}$ are the main objects of this paper. 
For the corresponding weights, on an arbitrary subset ${\tt D}\subset{\tt N}$ the following relation holds. 

\begin{property}\cite[Proposition 1]{V11}{\it Let ${\tt D}\subsetneq {\tt N}$ and ${\cal T}^\circ_{\tt D}\neq\emptyset$, then  
\begin{equation}
\lambda_{\tt D}^\circ =\min_{q\in {\tt D} }\left( \lambda_{\tt D}^{\bullet q} + \min_{r\notin{\tt D}}v_{qr}\right) . 
\label{lo}
\end{equation}}
 \end{property}
 
The calculation of the quantities $\lambda_{\tt D}^\circ$, $\lambda_{\tt D}^{\bullet}$ and $\lambda_{\tt D}^{\bullet q}$, as well as the construction of the corresponding minimal trees from $\tilde{\cal T}_{\tt D}^\circ$, $\tilde{\cal T}_{\tt D}^{\bullet}$ and $\tilde{\cal T}_{\tt D}^{\bullet q}$ is carried out using efficient algorithms \cite{V8,V10,V11}. 

We will use one simple property of the arc replacement operation \cite[Corollaries 1,2 from Lemma 1]{V6}, which we formulate more generally.

\begin{property}{\it Let $F$ and $G$ be two forests and ${\tt V}G\subseteq{\tt V}F$, and the set ${\tt D}\subseteq {\tt V}F\cap{\tt V}G$. Then the graph $F^G_{\uparrow{\tt D}}$ is a forest if any of the following holds: 

1) ${\tt N}^{in}_{\tt D}(F)=\emptyset$;  

2) ${\tt N}^{out}_{\tt D}(G)=\emptyset$. 
}
\end{property}

In \cite{V6} the concept of relatedness of two spanning forests with different numbers of roots was also introduced. 

A forest $F\in{\cal F}^{k+1}$ is called an ancestor of a forest $R\in{\cal F}^{k}$, and $R$ is a descendant of $F$ if the following conditions are satisfied. There exist two roots $x,y$ of forest $F$ such that: ${\tt K}_R={\tt K}_F\setminus \{ y \}$; $T^R_q=T^F_q$ for $q\in{\tt K}_R\setminus \{ x\}$; $R|_{{\tt V}T^F_x}=T^F_x$; $R|_{{\tt V}T^F_y}$ is a tree. Forests $F$ and $R$ are called related.

The main theorem proved there indicates the minimum possible changes that must be made to a minimum spanning forest with one number of trees to obtain a minimum spanning forest with a number of trees that differs by one:

\begin{property} \cite[Theorem 2 (on related forests)]{V6}{\it Let ${\cal F}^{k}\neq \emptyset$, then any forest in $\tilde{\cal F}^{k+1}$ has a descendant in $\tilde{\cal F}^{k}$ and vice versa --- any forest in $\tilde{\cal F}^{k}$ has an ancestor in $\tilde{\cal F}^{k+1}$.} 
\end{property} 
It is this property that forms the basis for creating efficient algorithms for constructing minimum spanning forests. \cite{V8,V10,V11}. 

Note that if in $F\in\tilde{\cal F}^{k+1}$ and from some of its vertices an arc emanates, then in any of its descendants $R\in \tilde{\cal F}^k$ an arc also emanates from this vertex.
  
We will assume that the original graph $V$ is sufficiently dense, in the sense that there is at least one spanning tree, that is, the set ${\cal F}^1$ of spanning forests consisting of one tree is not empty. Then the sets ${\cal F}^k$, $k\in \{1, 2, \ldots, N\}$ are not empty either.

It follows from Property 4 that 
\begin{property}\cite[Theorem 1]{V7}{\it The sequence of algebras $\mathfrak{A}_k$ is increasing
$$\{{\tt N},\emptyset\}=\mathfrak{A}_1 \subseteq \mathfrak{A}_2
\subseteq \cdots \subseteq \mathfrak{A}_{N-1} \subseteq \mathfrak{A}_N \ =2^{\tt N} ,$$ where $2^{\tt N}$ is the family of all subsets
of the set ${\tt N}$.}
\end{property}

Not all algebras are different and this depends on how many times in the system of convexity inequalities \cite{V6,V} 

\begin{equation}
\phi^{k-1}-\phi^k\ge\phi^k- \phi^{k+1} \label{convex}
\end{equation}
for different $k$ the equality sign and the strict inequality sign \cite[Theorem 1]{V6} occur.

\begin{property}\cite[Theorem 2 and Corollary 1]{V7}{\it
Let 
\begin{equation}
\phi^{k-1}-\phi^k =\phi^k- \phi^{k+1} , \label{equal}
\end{equation}
then
$\mathfrak{A}_k=\mathfrak{A}_{k+1}$, $\aleph_k^\bullet=\aleph_{k+1}^\bullet$ and
\begin{equation}
\tilde{\cal F}^{k-1}|_{{\uparrow}{\tt Y}}\subseteq \tilde{\cal F}^k|_{\uparrow{\tt Y}}
\supseteq\tilde{\cal F}^{k+1}|_{\uparrow{\tt Y}} \ ,  \ \  {\tt Y}\in\aleph_k  .  
\label{al}
\end{equation}}
Here the expression (\ref{al}) is given in a strengthened form compared to the original formulation. It is valid, since the proof of the corresponding theorem relied on the replacement of arcs emanating from the vertices of a certain set (whose entries might not belong to it).  
\end{property}

If in the system of convexity inequalities for some $k$ there is a strict inequality
\begin{equation}
\phi^{k-1}-\phi^k >\phi^k- \phi^{k+1} , \label{nonequal}
\end{equation}
then a number of properties are satisfied.

\begin{property}\cite[Claim 3]{V7} {\it 
Let (\ref{nonequal}) be satisfied, then the same number of arcs emanate from the vertices of any set ${\tt A}\in\mathfrak{A}_k$ in all forests from $\tilde{\cal F}^k$.}
\end{property}

\begin{property}\cite[Theorem 4]{V7} {\it 
Let (\ref{nonequal}) hold, then $|\aleph_k^\bullet|=k$.}
\end{property}

\begin{property}\cite[Claim 4]{V7} {\it 
Let (\ref{nonequal}) hold, then any tree of the forest $F\in\tilde{\cal F}^k$ contains exactly one labeled atom of the algebra $\mathfrak{A}^k$.}
\end{property}

\begin{property}\cite[Claim 6]{V7} {\it 
Let (\ref{nonequal}) hold, $F\in\tilde{\cal F}^k$, ${\tt Y}\in \aleph_k$, then there exists a forest $H\in\tilde{\cal F}^k$ such that $H|_{\uparrow{\tt Y}}=F|_{\uparrow{\tt Y}}$ and ${\tt N}^{in}_{\tt Y}(H)=\emptyset$. }
\end{property}

\begin{property}\cite[Theorem 5]{V7} {\it 
Let (\ref{nonequal}) hold and ${\tt A}\in\mathfrak{A}_k$, then the value of $\Upsilon^F_{\tt A}$ is the same for all $F\in\tilde{\cal F}^k$.}
\end{property}

\begin{property}\cite[Theorem 7]{V7} {\it 
Let (\ref{nonequal}) hold, $ F\in\tilde{\cal F}^k$, ${\tt M}\in \aleph_k^\bullet $, then 
$F|_{\uparrow{\tt Z}}=F|_{\tt Z}$ is a tree.}
\end{property}
Here we specifically noted that arcs do not emanate from the labeled atom, which was not included in the original formulation but was present in the corresponding proof.

\begin{property}\cite[Theorem 8]{V7} {\it 
Let (\ref{nonequal}) hold and $ F\in\tilde{\cal
F}^{k-1}$, then there exists a forest $P\in\tilde{\cal F}^{k}$ and a ${\tt Z}\in\aleph_k^\bullet$ such that $P|_{\uparrow{\overline{\tt Z}}}=F|_{\uparrow{\overline{\tt Z}}}$, and $F|_{\tt Z}$ is a tree.}
\end{property}

\begin{property}\cite[Corollary 3 from Theorem 8]{V7} {\it 
Let (\ref{nonequal}) hold. Then   
  
\begin{equation}
\tilde{\cal F}^{k-1}|_{\uparrow{\tt X}}\subseteq \tilde{\cal F}^k|_{\uparrow{\tt X}} \ ,  \ \  {\tt X}\in\aleph_k^\circ. 
\label{uat}
\end{equation}}
\end{property}
The formulation of this property is also given using the outgoing narrowing, that is, in a strengthened form (which was actually proven), compared to the original.

Finally, whether the convexity inequalities (\ref{convex}) are strict or non-strict, it holds that
\begin{property}\cite[Theorem 5,6]{V9} {\it Let $F\in\tilde{\cal F}^k \cup \tilde{\cal F}^{k-1}$, ${\tt Y}\in\aleph_k$, 
then $F|_{\tt Y}$ is a tree.} 
\end{property}

This property is the most important. Its proof is quite non-trivial and a separate paper \cite{V9} is devoted to it. It is this property that allows us to reduce the construction of minimal forests to the study of trees from the sets $\tilde{\cal T}^{\bullet}_{\tt Y}$, $\tilde{\cal T}^{\circ}_{\tt Y}$ on the atoms ${\tt Y}$ of the algebras $\mathfrak{A}_k$.

\section{How Trees Determine Forest Growth}
 
Summarizing the listed properties, we obtain the following picture.

First, note that, whether the convexity inequality (\ref{convex}) is strict or not, if ${\tt X}$ is an unlabeled atom, then there exist at least two atoms ${\tt Y},{\tt Y}'\in\aleph_k$ and at least two forests $F, F'\in\tilde{\cal F}^k$ such that in forest $F$, an arc outgoing from ${\tt X}$ enters ${\tt Y}$, and in forest $F'$, an arc outgoing from ${\tt X}$ enters ${\tt Y}'$. Moreover, there are at least two forests of $\tilde{\cal F}^k$ in which ${\tt X}$ is in trees rooted at distinct labeled atoms of the family $\aleph_k^\bullet$ (the converse would contradict ${\tt X}$ being an atom).

Further, in the forest $F\in\tilde{\cal F}^l$, $l\leq k$, from each vertex of the unlabeled atom ${\tt X}\in\aleph_k^\circ$ one arc emanates. If $F\in\tilde{\cal F}^k\cup \tilde{\cal F}^{k-1}$ from the set ${\tt X}$ itself there is exactly one arc emanating, and it is valid $F|_{\uparrow{\tt X}}\in{\cal T}^\circ_{\tt X}$.

Below are a number of simple theorems on the relationship between the properties of minimal trees on atoms of algebras of subsets of $\mathfrak{A}_k$ and the properties of minimal spanning forests. Together, these theorems and the examples given completely exhaust all possible variants.

\subsection{In convexity inequalities, one sign of strict inequality is known}

Now let the convexity inequalities have the strict inequality sign. First of all, note that when (\ref{nonequal}) holds, the arcs of the labeled atoms of the algebra $\mathfrak{A}_k$ do not emanate in any forest from $\tilde{\cal F}^k$.

\begin{theorem}
Let (\ref{nonequal}) hold and ${\tt Z}\in \aleph_k^\bullet$. Then $\tilde{\cal F}^k|_{\uparrow{\tt Z}}=\tilde{\cal F}^k|_{\tt Z} =\tilde{\cal T}^\bullet_{\tt Z}$. 
 
\end{theorem}
\begin{proof}
Let $F\in\tilde{\cal F}^k$. By Property 10, there exists a forest $H\in\tilde{\cal F}^k$ such that $H|_{\uparrow{\tt Z}}=F|_{\uparrow{\tt Z}}$ and ${\tt N}^{in}_{\tt Z}(H)=\emptyset$. By Property 12, $F|_{\uparrow{\tt Z}}=F|_{\tt Z}$. Thus, ${\tt Z}$ is the vertex set of one of the trees in forest $H$. Then by Property 1 $H|_{\tt Z}\in\tilde{\cal T}^\bullet_{\tt Z}$ ($\Upsilon^H_{\tt Z}=\lambda^\bullet_{\tt Z}$), and hence $F|_{\tt Z}\in\tilde{\cal T}^\bullet_{\tt Z}$. 
 
 Now let $T\in \tilde{\cal T}^\bullet_{\tt Z}$. By Property 3, the graph $G=H^T_{\uparrow{\tt Z}}$ is a forest. It has the same number of trees as $H$, since ${\tt Z}$ is the vertex set of one of the trees in $H$. Therefore, $G\in{\cal F}^k$. Since $H$ is a minimal forest, $\Upsilon^G-\Upsilon^H \geq 0$. Given that $G|_{\uparrow{\tt Z}}=G|_{\tt Z}=T$, and $\Upsilon^T=\lambda^\bullet_{\tt Z}$, we have

\begin{equation}
0\leq \Upsilon^G-\Upsilon^H=\Upsilon^G_{\tt Z}-\Upsilon^H_{\tt Z}= \lambda^\bullet_{\tt Z}-\lambda^\bullet_{\tt Z}=0 \ .
\label{lg}
\end{equation} 
Thus $G\in\tilde{\cal F}^k$.
\end{proof}

Unlabeled atoms have a similar property.

\begin{theorem} Let (\ref{nonequal}) be satisfied and ${\tt X}\in\aleph^\circ_k$. 
Then $\tilde{\cal F}^k|_{\uparrow{\tt X}}=\tilde{\cal T}^\circ_{\tt X}$. 
\end{theorem}

\begin{proof}
Let $T\in \tilde{\cal T}^\circ_{\tt X}$ and $F\in\tilde{\cal F}^k$. By Property 10, there is a forest $H\in\tilde{\cal F}^k$ such that $H|_{\uparrow{\tt X}}=F|_{\uparrow{\tt X}}$ and ${\tt N}^{in}_{\tt X}(H)=\emptyset$. By Property 3, the graph $G=H^T_{\uparrow{\tt X}}$ is a forest. In both $H$ and $T$, arcs emanate from all vertices of ${\tt X}$. Therefore, the number of trees in $G$ is the same as in $H$, and hence $G\in{\cal F}^k$. 
We have 

\begin{equation}
0\leq \Upsilon^G-\Upsilon^H=\Upsilon^G_{\tt X}-\Upsilon^H_{\tt X}=\Upsilon^T-\Upsilon^F_{\tt X} =\lambda^\circ_{\tt X}-\Upsilon^H_{\tt X} .
\label{lg1}
\end{equation} 
By Property 15, $F|_{\tt X}$ is a tree, so $F|_{\uparrow{\tt X}}\in{\cal T}^\circ_{\tt X}$, that is, $H|_{\uparrow{\tt X}}\in{\cal T}^\circ_{\tt X}$. Hence $\Upsilon^H_{\tt X}\geq \lambda^\circ_{\tt X}$. Then it follows from (\ref{lg1}) that $\Upsilon^H_{\tt X} = \lambda^\circ_{\tt X}$ and that $G\in\tilde{\cal F}^k$. Moreover, $H|_{\uparrow{\tt X}}=F|_{\uparrow{\tt X}}$. Therefore $F|_{\uparrow{\tt X}}\in\tilde{\cal T}^\circ_{\tt X}$ and by construction $G|_{\uparrow{\tt X}}=T$.
\end{proof}

Let us now turn to forests from $\tilde{\cal F}^{k-1}$. The simplest situation is with unlabeled atoms of the algebra $\mathfrak{A}_k$.

\begin{theorem}
Let (\ref{nonequal}) hold and ${\tt X}\in\aleph^\circ_k$. Then $\tilde{\cal F}^{k-1}|_{\uparrow{\tt X}}\subseteq \tilde{\cal T}^\circ_{\tt X}$.
\end{theorem}

\begin{proof}
Indeed, by Property 14, $\tilde{\cal F}^{k-1}|_{\uparrow{\tt X}}\subseteq \tilde{\cal F}^k|_{\uparrow{\tt X}}$. By Theorem 2, $\tilde{\cal F}^k|_{\uparrow{\tt X}}=\tilde{\cal T}^\circ_{\tt X}$. 
\end{proof}

Thus, the distribution of arcs emanating from the vertices of any unlabeled atom in forests of $\tilde{\cal F}^{k-1}$ has nothing new compared to the distribution of arcs emanating from these vertices in forests of $\tilde{\cal F}^{k}$. And even the variability may decrease. 

For labeled atoms of the algebra $\mathfrak{A}_k$ the situation is much more diverse. 
 
\begin{theorem} Let (\ref{nonequal}) hold and $F\in\tilde{\cal F}^{k-1}$. Then there exists a unique labeled atom ${\tt Z}\in\aleph^\bullet_k$ such that ${\tt N}^{out}_{\tt Z}(F)\neq \emptyset$. For this atom, $F|_{\uparrow{\tt Z}}\in\tilde{\cal T}^\circ_{\tt Z}$. Moreover, for any $T\in\tilde{\cal T}^\circ_{\tt Z}$, there exists $G\in\tilde{\cal F}^{k-1}$ such that $G|_{\uparrow{\tt Z}}=T$.  
\end{theorem}
\begin{proof}
First, let us verify that for every $F\in\tilde{\cal F}^{k-1}$, there is a unique labeled atom from which an arc in the forest $F$ emanates. By Property 13, for $F$ there is a forest $P\in\tilde{\cal F}^{k}$ and a ${\tt Z}\in\aleph_k^\bullet$ such that $P|_{\uparrow{\overline{\tt Z}}}=F|_{\uparrow{\overline{\tt Z}}}$. That is, both $F$ and $P$ have the same number of arcs emanating from the vertices of $\overline{\tt Z}$. Since the forest $P$ does not have an arc from ${\tt Z}$, and the forest $F$ has one more arc, then the atom ${\tt Z}$ has an arc from $F$. All other atoms of the family $\aleph^\bullet_k$ in the forest $F$ do not have arcs from them, since they do not come from the forest $P$. Thus, the uniqueness of such a marked atom ${\tt Z}$ holds. 

 By Property 10, there exists a forest $H\in\tilde{\cal F}^k$ such that ${\tt N}^{in}_{\tt Z}(H)=\emptyset$ and $H|_{\uparrow{\tt Z}}=P|_{\uparrow{\tt Z}}$. 
Let us introduce the graph $Q=H^F_{\uparrow{\tt Z}}$, which by Property 3 is a forest (by the way, by construction it is a descendant of the forest $H$).
For it $Q|_{\uparrow{\tt Z}}=F|_{\uparrow{\tt Z}}$ and ${\tt N}^{in}_{\tt Z}(Q)=\emptyset$. It is also easy to verify that $Q\in\tilde{\cal F}^{k-1}$. Indeed, given that $\Upsilon^Q_{\overline{\tt Z}}=\Upsilon^H_{\overline{\tt Z}}$, we have

\begin{equation*}
\Upsilon^Q= \Upsilon^F_{\tt Z}+\Upsilon^H_{\overline{\tt Z}}= \Upsilon^F_{\tt Z}+\Upsilon^P_{\overline{\tt Z}}= \Upsilon^F_{\tt Z}+\Upsilon^F_{\overline{\tt Z}}=\Upsilon^F.
\end{equation*}
Here the first equality is by construction, the second by Property 11, the third by Property 13.

Let's represent the weight of the forest $Q$ as

\begin{equation}
\phi^{k-1}=\Upsilon^Q=\Upsilon^Q_{\tt Z}+\Upsilon^Q_{\overline{\tt Z}}  =\Upsilon^Q_{\tt Z}+\Upsilon^H_{\overline{\tt Z}} . 
\label{k-1}
\end{equation}

Now let $T\in\tilde{\cal T}^\circ_{\tt Z}$ and $G=H^T_{\uparrow{\tt Z}}$. Since ${\tt N}^{in}_{\tt Z}(H)=\emptyset$, then by Property 3 $G$ is a forest and, since $G$ has one more arc than $H$, then $G\in{\cal F}^{k-1}$. Let us verify that $G$ is a minimal forest. We have

\begin{equation}
\phi^{k-1}\leq \Upsilon^G= \Upsilon^G_{\tt Z}+\Upsilon^H_{\overline{\tt Z}}=\Upsilon^T+\Upsilon^H_{\overline{\tt Z}}=\lambda^\circ_{\tt Z}+\Upsilon^H_{\overline{\tt Z}} . 
\label{k-1'}
\end{equation}
From (\ref{k-1}) and (\ref{k-1'}), we obtain

\begin{equation}
\Upsilon^Q_{\tt Z}\leq \lambda^\circ_{\tt Z} \ .
\end{equation}
However, $Q|_{\uparrow{\tt Z}}\in {\cal T}^\circ_{\tt Z}$. This means that $\Upsilon^Q_{\tt Z}\geq \lambda^\circ_{\tt Z}$. That is, $\Upsilon^Q_{\tt Z}= \lambda^\circ_{\tt Z}$ and $Q|_{\uparrow{\tt Z}}\in \tilde{\cal T}^\circ_{\tt Z}$, and therefore $F|_{\uparrow{\tt Z}}\in \tilde{\cal T}^\circ_{\tt Z}$. Now $\Upsilon^G=\phi^{k-1}$, that is, $G\in\tilde{\cal F}^{k-1}$, and by construction $G|_{\uparrow{\tt Z}}=T$. 
\end{proof}

Note that Theorem 4 does not mean that $\tilde{\cal F}^{k-1}|_{\uparrow{\tt Z}}=\tilde{\cal T}^\circ_{\tt Z}$. The point is that for two forests $F$ and $F'$ from $\tilde{\cal F}^{k-1}$, the marked atoms ${\tt Z}$ and ${\tt Z}'$ that have an outgoing arc in them are, generally speaking, different. Let us clarify the possible situations for marked atoms.

\begin{theorem}
Let (\ref{nonequal}) hold and let there exist an atom ${\tt Z}\in\aleph^\bullet_k$ such that for any forest $F\in\tilde{\cal F}^{k-1}$ ${\tt N}^{out}_{\tt Z}(F)\neq \emptyset$ holds. Then $\tilde{\cal F}|_{\uparrow{\tt Z}}=\tilde{\cal T}^\circ_{\tt Z}$ and

\begin{equation}
\phi^{k-2}- \phi^{k-1} > \phi^{k-1}-\phi^k > \phi^k- \phi^{k+1}  .  
\label{ineq}
\end{equation}

\end{theorem} 

\begin{proof}

First of all, we note that if such an atom ${\tt Z}$ exists, then it is the only one satisfying the conditions of the theorem. Indeed, according to Theorem 4, for any $F\in\tilde{\cal F}^{k-1}$ there is a unique ${\tt Z}\in\aleph^\bullet_k$ such that ${\tt N}^{out}_{\tt Z}(F)\neq\emptyset$. According to the condition, this atom is the same for all $F$. But then from the same Theorem 4 it follows that for any $F\in\tilde{\cal F}^{k-1}$, $F|_{\uparrow{\tt Z}}\in\tilde{\cal T}^\circ_{\tt Z}$ holds. And also, for any $T\in\tilde{\cal T}^\circ_{\tt Z}$, there exists $F\in\tilde{\cal F}^{k-1}$ such that $F|_{\uparrow{\tt Z}}=T$. That is, $\tilde{\cal F}|_{\uparrow{\tt Z}}=\tilde{\cal T}^\circ_{\tt Z}$. 

Further, it follows from the condition that ${\tt Z}$ does not belong to the family $\aleph_{k-1}^\bullet$. This means that $\aleph_{k-1}^\bullet \neq \aleph_k^\bullet$. Thus, by Property 6, the equality

\begin{equation*}
\phi^{k-2}- \phi^{k-1} = \phi^{k-1}-\phi^k    
\end{equation*}
cannot be completed.
\end{proof}
 
Note that from the execution of (\ref{ineq}) it does not at all follow that a specific marked atom ${\tt Z}\in\aleph^\bullet_k$, "servicing" \ all forests from $\tilde{\cal F}^{k-1}$ at once, exists. In the example in Fig.\ref{p5} there is no such atom.

\begin{figure}[h]
\unitlength=0.9mm
\begin{center}
\begin{picture}(147,35)

\put(12,32){$V$}
\put(4,8){\vector(0,1){15}}
\put(21,6){\vector(-1,0){16}} 
\put(23,23){\vector(0,-1){15}}
\put(3,23){\vector(0,-1){15}} 
\put(5,25){\vector(1,0){16}}
\put(5,15){\small 2}
\put(0,15){\small 2}
\put(20,15){\small 3}
\put(12,1){\small 1}
\put(12,21){\small 3}
\put(1,5){$\odot$}
\put(21,5){$\odot$}
\put(1,24){$\odot$}
\put(21,24){$\odot$}
\put(4,27){$\tt J$}
\put(19,27){$\tt I$}
\put(20,0){$\tt L$}
\put(4,0){$\tt L'$}

\put(42,32){$F$}
\put(33,8){\vector(0,1){15}}
\put(53,23){\vector(0,-1){15}}
\put(51,6){\vector(-1,0){16}}
\put(34,15){\small 2}
\put(50,15){\small 3}
\put(42,1){\small 1}
\put(31,5){$\odot$}
\put(51,5){$\odot$}
\put(31,24){$\odot$}
\put(51,24){$\odot$}
\put(50,0){$\tt L$}
\put(34,0){$\tt L'$}
\put(35,27){$\tt J$}
\put(49,27){$\tt I$}
\put(42,8){\tt Z}
\put(43,6){\oval(26,14)}

\put(72,32){$G$}
\put(83,23){\vector(0,-1){15}}
\put(63,23){\vector(0,-1){15}}
\put(81,6){\vector(-1,0){16}}
\put(60,15){\small 2}
\put(72,1){\small 1}
\put(80,15){\small 3}
\put(61,5){$\odot$}
\put(81,5){$\odot$}
\put(61,24){$\odot$}
\put(81,24){$\odot$}
\put(65,27){$\tt J$}
\put(79,27){$\tt I$}
\put(80,0){$\tt L$}
\put(64,0){$\tt L'$}
\put(72,8){\tt Z}
\put(73,6){\oval(26,14)}

\put(102,32){$H$}
\put(93,8){\vector(0,1){15}}
\put(95,25){\vector(1,0){16}}
\put(111,6){\vector(-1,0){16}}
\put(90,15){\small 2}
\put(102,21){\small 3}
\put(102,1){\small 1}
\put(91,5){$\odot$}
\put(111,5){$\odot$}
\put(91,24){$\odot$}
\put(111,24){$\odot$}
\put(95,27){$\tt J$} 
\put(109,27){$\tt I$}
\put(110,0){$\tt L$}
\put(94,0){$\tt L'$}
\put(102,8){\tt Z}
\put(103,06){\oval(26,14)} 

\put(132,32){$Q$}
\put(141,6){\vector(-1,0){16}}
\put(132,1){\small 1}
\put(121,5){$\odot$}
\put(141,5){$\odot$}
\put(121,24){$\odot$}
\put(141,24){$\odot$}
\put(125,27){$\tt J$} 
\put(139,27){$\tt I$}
\put(140,0){$\tt L$}
\put(124,0){$\tt L'$}
\put(132,8){\tt Z}
\put(133,06){\oval(26,14)}

\end{picture} 
\caption{\small Depicted from left to right: graph $V$; forests $F$, $G$, $H$, which form the set $\tilde{\cal F}^1={\cal F}^1$ (in this case they are spanning trees); The forest on the far right is the forest $Q$, which forms the set $\tilde{\cal F}^3$. Here $\phi^4=0$, $\phi^3=1$, $\phi^2=3$, $\phi^1=6$. The following is satisfied: $\phi^1-\phi^2>\phi^2-\phi^3> \phi^3-\phi^4$. The labeled atoms, which form $\aleph^\bullet_3$, are the vertices ${\tt I}$, ${\tt J}$ of graph $V$ and the set ${\tt Z}=\{ {\tt L,L'}\}$.
Among them there is no atom from which an arc would emanate in all the forests of the set ${\tilde{\cal F}^1}$.}
\label{p5}
\end{center}
\end{figure}
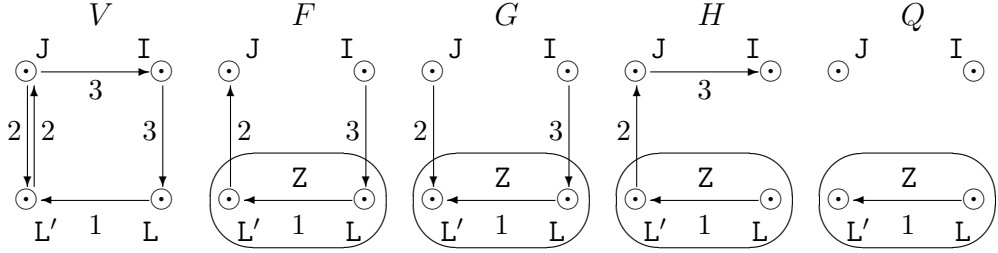 

 The situation itself, when (\ref{ineq}) is satisfied, can naturally be considered a "general position" \ situation. In \cite{VF}, the "general position" \ case \ was called an even more structured variant, when the minimum weight of forests from ${\cal F}^k$ for each $k$ was achieved on a single graph. Such a situation almost certainly occurs if the arcs of the original graph are random numbers.  
 
Let's see what it means if, on the contrary, there exists such ${\tt Z}\in\aleph^\bullet_k$ that no arc emanates from it in any forest $ F\in\tilde{\cal F}^{k-1}$.
 
 \begin{theorem}
Let (\ref{nonequal}) hold, and there exists a 
labeled atom ${\tt Z}\in\aleph^\bullet_k$ such that ${\tt N}^{out}_{\tt Z}(\tilde{\cal F}^{k-1})= \emptyset$. 
Then  $ \tilde{\cal F}^{k-1}|_{\uparrow{\tt Z}}=\tilde{\cal F}^{k-1}|_{\tt Z}=\tilde{\cal T}^\bullet_{\tt Z}$.  
\end{theorem}

\begin{proof}
Let $F\in\tilde{\cal F}^{k-1}$. By Theorem 4, there exists a unique atom ${\tt L}\in\aleph_k^\bullet$ from which an arc emanates in it. The atom ${\tt L}$ does not coincide with ${\tt Z}$ by the assumption, so ${\tt M}\subset\overline{\tt L}$. By Property 13, there exists a forest $P\in\tilde{\cal F}^k$ such that $P|_{\uparrow\overline{\tt L}}=F|_{\uparrow\overline{\tt L}}$. In particular, $P|_{\uparrow{\tt Z}}=F|_{\uparrow{\tt Z}}$. And for $P$, by Theorem 1, $P|_{\uparrow{\tt Z}}= P|_{\tt Z}\in\tilde{\cal T}^\bullet_{\tt Z}$. Therefore, $F|_{\uparrow{\tt Z}} =F|_{\tt Z}\in\tilde{\cal T}^\bullet_{\tt Z}$. 

Now let $T\in\tilde{\cal T}^\bullet_{\tt Z}$ and $F\in \tilde{\cal F}^{k-1}$. Since in $T$ there is no arc coming out of ${\tt Z}$, then by Property 3 $Q=F^T_{\uparrow{\tt Z}}$ is a forest. In the forest $F$ there is no arc coming out of ${\tt Z}$ either, therefore $Q\in{\cal F}^{k-1}$. Let us verify that $Q$ is minimal. Indeed,
\begin{equation*}
\Upsilon^Q - \Upsilon^F=\Upsilon^Q_{\tt Z}+\Upsilon^Q_{\overline{\tt Z}}-(\Upsilon^F_{\tt Z}+\Upsilon^F_{\overline{\tt Z}})=\lambda^\bullet_{\tt Z} +\Upsilon^F_{\overline{\tt Z}}-(\lambda^\bullet_{\tt Z} +\Upsilon^F_{\overline{\tt Z}})=0. 
\end{equation*}

Thus, $Q\in\tilde{\cal F}^{k-1}$, and by construction $Q|_{\uparrow{\tt Z}}= Q|_{\tt Z}=T$. 
\end{proof}

There remains an intermediate variant, when for ${\tt Z}\in\aleph^\bullet_k$ there is both a forest in which an arc emanates from ${\tt Z}$, and a forest in which an arc does not emanate from ${\tt Z}$.

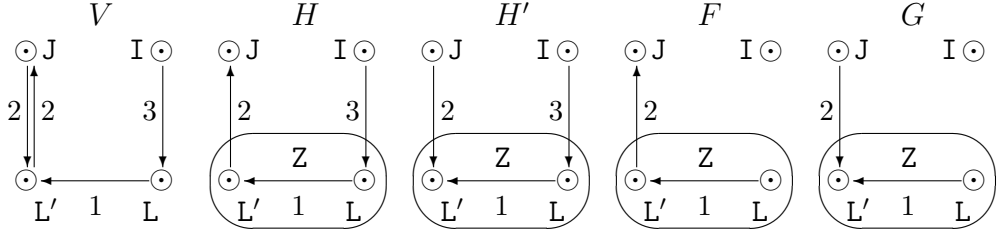
\begin{figure}[h]
\unitlength=0.9mm
\begin{center}
\begin{picture}(147,32)

\put(12,29){$V$}
\put(4,8){\vector(0,1){15}}
\put(21,6){\vector(-1,0){16}} 
\put(23,23){\vector(0,-1){15}}
\put(3,23){\vector(0,-1){15}} 
\put(5,15){\small 2}
\put(0,15){\small 2}
\put(20,15){\small 3}
\put(12,1){\small 1}
\put(1,5){$\odot$}
\put(21,5){$\odot$}
\put(1,24){$\odot$}
\put(21,24){$\odot$}
\put(5,24){$\tt J$}
\put(18,24){$\tt I$}
\put(20,0){$\tt L$}
\put(4,0){$\tt L'$}

\put(42,29){$H$}
\put(33,8){\vector(0,1){15}}
\put(53,23){\vector(0,-1){15}}
\put(51,6){\vector(-1,0){16}}
\put(34,15){\small 2}
\put(50,15){\small 3}
\put(42,1){\small 1}
\put(31,5){$\odot$}
\put(51,5){$\odot$}
\put(31,24){$\odot$}
\put(51,24){$\odot$}
\put(50,0){$\tt L$}
\put(34,0){$\tt L'$}
\put(35,24){$\tt J$}
\put(48,24){$\tt I$}
\put(42,8){\tt Z}
\put(43,6){\oval(26,14)}

\put(72,29){$H'$}
\put(63,23){\vector(0,-1){15}}
\put(83,23){\vector(0,-1){15}}
\put(81,6){\vector(-1,0){16}}
\put(64,15){\small 2}
\put(80,15){\small 3}
\put(72,1){\small 1}
\put(61,5){$\odot$}
\put(81,5){$\odot$}
\put(61,24){$\odot$}
\put(81,24){$\odot$}
\put(80,0){$\tt L$}
\put(64,0){$\tt L'$}
\put(65,24){$\tt J$}
\put(78,24){$\tt I$}
\put(72,8){\tt Z}
\put(73,6){\oval(26,14)}

\put(102,29){$F$}
\put(93,8){\vector(0,1){15}}
\put(111,6){\vector(-1,0){16}}
\put(94,15){\small 2}
\put(102,1){\small 1}
\put(91,5){$\odot$}
\put(111,5){$\odot$}
\put(91,24){$\odot$}
\put(111,24){$\odot$}
\put(95,24){$\tt J$}
\put(108,24){$\tt I$}
\put(110,0){$\tt L$}
\put(94,0){$\tt L'$}
\put(102,8){\tt Z}
\put(103,6){\oval(26,14)}

\put(132,29){$G$}
\put(123,23){\vector(0,-1){15}}
\put(141,6){\vector(-1,0){16}}
\put(120,15){\small 2}
\put(132,1){\small 1}
\put(121,5){$\odot$}
\put(141,5){$\odot$}
\put(121,24){$\odot$}
\put(141,24){$\odot$}
\put(125,24){$\tt J$} 
\put(138,24){$\tt I$}
\put(140,0){$\tt L$}
\put(124,0){$\tt L'$}
\put(132,8){\tt Z}
\put(133,06){\oval(26,14)}

\end{picture} 
\caption{\small Depicted from left to right: graph $V$; forests $H$ and $H'$, which form the set $\tilde{\cal F}^1={\cal F}^1$; forests $F$ and $G$, which form the set ${\tilde{\cal F}^2}$. The labeled atoms, which form $\aleph^\bullet_3$, are the vertices ${\tt I}$, ${\tt J}$ of graph $V$ and the set ${\tt Z}=\{ {\tt L,L'}\}$. Here $\phi^4=0$, $\phi^3=1$, $\phi^2=3$, $\phi^1=6$. The following holds: $\phi^1-\phi^2>\phi^2-\phi^3> \phi^3-\phi^4$. In this case ${\tt N}^{out}_{\tt Z}(F)=\{{\tt J}\}\neq \emptyset$ and ${\tt N}^{out}_{\tt Z}(G)= \emptyset$.}
\label{>>}
\end{center}
\end{figure}

\begin{theorem} 
Let (\ref{nonequal}) hold, and there exist a 
marked atom ${\tt Z}\in\aleph^\bullet_k$ and forests $F,G \in\tilde{\cal F}^{k-1}$ such that ${\tt N}^{out}_{\tt Z}(F)\neq \emptyset$ and ${\tt N}^{out}_{\tt Z}(G)= \emptyset$. Then  
 
1) $F|_{\uparrow{\tt Z}}\in\tilde{\cal T}^\circ_{\tt Z}$. For any $T\in\tilde{\cal T}^\circ_{\tt Z}$, there exists $Q\in\tilde{\cal F}^{k-1}$ such that $Q|_{\uparrow{\tt Z}}=T$; 

2) $G|_{\tt Z}\in\tilde{\cal T}^\bullet_{\tt Z}$. For any $T\in\tilde{\cal T}^\bullet_{\tt Z}$, there exists $Q\in\tilde{\cal F}^{k-1}$ such that $Q|_{\uparrow{\tt Z}}=Q|_{\tt Z}=T$. 
\end{theorem} 

\begin{proof}

1) Since there exists a forest $F\in\tilde{\cal F}^{k-1}$ in which an arc emanates from ${\tt Z}$, then by Theorem 4 $F|_{\uparrow{\tt Z}}\in\tilde{\cal T}^\circ_{\tt Z}$, and for any $T\in\tilde{\cal T}^\circ_{\tt Z}$, there exists $Q\in\tilde{\cal F}^{k-1}$ such that $Q|_{\uparrow{\tt Z}}=T$. 

$2)$ Now let $G\in\tilde{\cal F}^{k-1}$ be such that ${\tt N}_{\tt Z}^{out}(G)=\emptyset$. Then, by Theorem 4, there exists a unique marked atom ${\tt L}\in\aleph_{k}^\bullet$ (obviously not coinciding with ${\tt Z}$) from which an arc emanates. Now, just as in Theorem 6, we verify that $G|_{\tt Z}\in\tilde{\cal T}^\bullet_{\tt Z}$ and that for any $T\in\tilde{\cal T}^\bullet_{\tt Z}$, there exists $Q\in\tilde{\cal F}^{k-1}$ such that $Q|_{\uparrow{\tt Z}}=Q|_{\tt Z}=T$.
\end{proof}

Satisfaction of the conditions of Theorem 7 does not indicate a further sign in the convexity inequalities. In Fig.\ref{>>} (\ref{ineq}) is satisfied, whereas in Fig.\ref{=>} $\phi^{k-2}-\phi^{k-1}=\phi^{k-1}-\phi^k> \phi^k-\phi^{k+1}$ is satisfied.
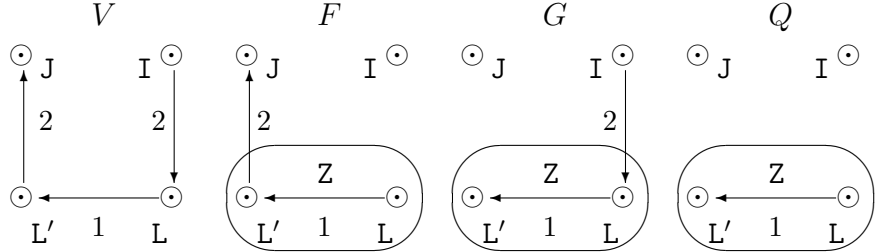
\begin{figure}[h]
\unitlength=1mm
\begin{center}
\begin{picture}(115,32)

\put(12,29){$V$}
\put(3,8){\vector(0,1){15}}
\put(21,6){\vector(-1,0){16}} 
\put(23,23){\vector(0,-1){15}}
\put(5,15){\small 2}
\put(20,15){\small 2}
\put(12,1){\small 1}
\put(1,5){$\odot$}
\put(21,5){$\odot$}
\put(1,24){$\odot$}
\put(21,24){$\odot$}
\put(5,22){$\tt J$}
\put(18,22){$\tt I$}
\put(20,0){$\tt L$}
\put(4,0){$\tt L'$}

\put(42,29){$F$}
\put(33,8){\vector(0,1){15}}
\put(51,6){\vector(-1,0){16}}
\put(34,15){\small 2}
\put(42,1){\small 1}
\put(31,5){$\odot$}
\put(51,5){$\odot$}
\put(31,24){$\odot$}
\put(51,24){$\odot$}
\put(50,0){$\tt L$}
\put(34,0){$\tt L'$}
\put(35,22){$\tt J$}
\put(48,22){$\tt I$}
\put(42,8){\tt Z}
\put(43,6){\oval(26,14)}

\put(72,29){$G$}
\put(83,23){\vector(0,-1){15}}
\put(81,6){\vector(-1,0){16}}
\put(72,1){\small 1}
\put(80,15){\small 2}
\put(61,5){$\odot$}
\put(81,5){$\odot$}
\put(61,24){$\odot$}
\put(81,24){$\odot$}
\put(65,22){$\tt J$}
\put(78,22){$\tt I$}
\put(80,0){$\tt L$}
\put(64,0){$\tt L'$}
\put(72,8){\tt Z}
\put(73,6){\oval(26,14)}

\put(102,29){$Q$}
\put(111,6){\vector(-1,0){16}}
\put(102,1){\small 1}
\put(91,5){$\odot$}
\put(111,5){$\odot$}
\put(91,24){$\odot$}
\put(111,24){$\odot$}
\put(95,22){$\tt J$} 
\put(108,22){$\tt I$}
\put(110,0){$\tt L$}
\put(94,0){$\tt L'$}
\put(102,8){\tt Z}
\put(103,06){\oval(26,14)} 

\end{picture} 
\caption{\small Depicted from left to right: graph $V$, coinciding with the unique spanning forest of the set $\tilde{\cal F}^1={\cal F}^1$; forests $F$ and $G$, constituting the set $\tilde{\cal F}^2$; forest $Q$, constituting the set $\tilde{\cal F}^3$. The labeled atoms constituting $\aleph^\bullet_3$ are vertices ${\tt I}$, ${\tt J}$ of graph $V$ and the set ${\tt Z}=\{ {\tt L,L'}\}$. Here $\phi^4=0$, $\phi^3=1$, $\phi^2=3$, $\phi^1=5$. Satisfied $\phi^1-\phi^2=\phi^2-\phi^3> \phi^3-\phi^4$. In this case ${\tt N}^{out}_{\tt Z}(F)=\{{\tt J}\}\neq \emptyset$ and ${\tt N}^{out}_{\tt Z}(G)= \emptyset$.  }
\label{=>}
\end{center}
\end{figure} 

\begin{conclusion}
Note that Theorems 1 and 2 mean, in particular, the following. Suppose that when (\ref{nonequal}) holds, all atoms from $\aleph_k$ are known, and also for every ${\tt Z}\in\aleph^\bullet_k$ the set of trees $\tilde{\cal T}^\bullet_{\tt Z}$ is known, and for every ${\tt X}\in\aleph^\circ_k$ the trees $\tilde{\cal T}^\circ_{\tt X}$ are known. This allows us to construct all forests from $\tilde{\cal F}^k$, as well as vice versa --- the set $\tilde{\cal F}^k$ completely determines both $\aleph_k$ and the sets $\tilde{\cal T}^\bullet_{\tt Z}$, $\tilde{\cal T}^\circ_{\tt X}$. If the trees $\tilde{\cal T}^\circ_{\tt Z}$ are also known, then according to Theorems 3-7, all forests of $\tilde{\cal F}^{k-1}$ can be constructed. In this way, we can define both the algebra $\mathfrak{A}_{k-1}$ and its labeled and unlabeled atoms and, accordingly, the sets of trees $\tilde{\cal T}^\bullet_{\tt Z'}$ and $\tilde{\cal T}^\circ_{\tt X'}$ for labeled ${\tt Z'}\in\aleph^\bullet_{k-1}$ and unlabeled ${\tt X'}\in\aleph^\circ_{k-1}$ atoms of the new algebra $\mathfrak{A}_{k-1}$.
However, this knowledge does not allow us to define trees from $\tilde{\cal T}^\circ_{\tt Z'}$ (they must be constructed anew). Thus, forests from $\tilde{\cal F}^{k-2}$ are also not defined. 

By Property 15, we know that for $F\in\tilde{\cal F}^{k-2}$, the graph $F|_{\tt Y'}$ is a tree, if ${\tt Y'}\in\aleph_{k-1}$. But if ${\tt Y}\in\aleph_{k}$, then it is not at all necessary that the graph $F|_{\tt Y}$ is a tree.

To be quite precise, note that for $F\in\tilde{\cal F}^{k-2}$, if (\ref{ineq}) holds by Theorem 3 for all ${\tt X'}\in\aleph^\circ_{k-1}$, $F|_{\uparrow{\tt X'}}$ is a tree from the set $\tilde{\cal T}^\circ_{\tt X'}$. And this set is constructed by us. Therefore, for ${\tt Y}\in\aleph_k$, if ${\tt Y}\subset{\tt X'}$, then $F|_{\uparrow{\tt Y}}$ is also a tree and it is known to us.

But if ${\tt Y}\subset{\tt Z'}\in\aleph^\bullet_{k-1}$, then by Property 13 we know only that $F|_{\tt Z'}$ is a tree and nothing more. For a tree $T\in \tilde{\cal T}^\circ_{\tt Z'}$, there is no reason to expect that the graph $T|_{\tt Y}$ is also a tree. It is, generally speaking, a forest. And the following example demonstrates this.
\end{conclusion}

Let us use the example from \cite{V7,V9} from this point of view. 

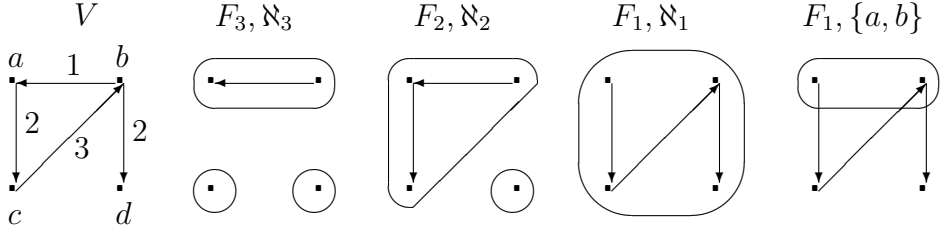
\begin{figure}[h]
\unitlength=1.1mm
\begin{center}
\begin{picture}(115,30)

\put (8,25){$V$}
\put(0,5){$\centerdot$}
\put(13,5){$\centerdot$}
\put(0,18){$\centerdot$}
\put(13,18){$\centerdot$}
\put(0,20){$a$}
\put(13,20){$b$}
\put(0,1){$c$}
\put(13,1){$d$}
\put(1,18){\vector(0,-1){12}}
\put(13,18){\vector(-1,0){12}}
\put(14,18){\vector(0,-1){12}}
\put(1,5){\vector(1,1){13}}
\put(7,19){1}
\put(2,12){2}
\put(8,9){3}
\put(15,11){2}

\put(25,25){$F_3, \aleph_3$}
\put(24,5){$\centerdot$}
\put(37,5){$\centerdot$}
\put(24,18){$\centerdot$}
\put(37,18){$\centerdot$}
\put(37,18){\vector(-1,0){12}}
\put(31,18){\oval(17,6)}
\put(25,5){\circle{5}}
\put(37,5){\circle{5}}

\put (49,25){$F_2,\aleph_2$}
\put(48,5){$\centerdot$}
\put(61,5){$\centerdot$}
\put(48,18){$\centerdot$}
\put(61,18){$\centerdot$}
\put(61,18){\vector(-1,0){12}}
\put(49,18){\vector(0,-1){12}}
\put(49,18){\oval(6,6)[tl]}
\put(46,18){\line(0,-1){12}}
\put(49,6){\oval(6,6)[bl]}
\put(49,21){\line(1,0){12}}
\put(61,18){\oval(6,6)[tr]}
\put(49,3){\line(1,1){15}}
\put(61,5){\circle{5}}

\put(73,25){$F_1,\aleph_1$}
\put(72,5){$\centerdot$}
\put(85,5){$\centerdot$}
\put(72,18){$\centerdot$}
\put(85,18){$\centerdot$}
\put(73,18){\vector(0,-1){12}}
\put(86,18){\vector(0,-1){12}}
\put(73,5){\vector(1,1){13}}
\put(79,12){\oval(20,20)}

\put (96,25){$F_1,\{a,b\}$}
\put(97,5){$\centerdot$}
\put(110,5){$\centerdot$}
\put(97,18){$\centerdot$}
\put(110,18){$\centerdot$}
\put(98,18){\vector(0,-1){12}}
\put(111,18){\vector(0,-1){12}}
\put(98,5){\vector(1,1){13}}
\put(104,18){\oval(17,6)}

\end{picture}  
\caption{\small Graph $V$ and its arc weights (left), minimal forests $F_{3,2,1}$ and atoms of algebras $\mathfrak{A}_{3,2,1}$. On the right, the forest $F_1$ restricted to a labeled atom $\{a,b\}$ of the algebra of subsets $\mathfrak{A}_3$ is not a tree.}
\label{atom}
\end{center}
\end{figure}

\begin{example}
The weights of the arcs of the graph $V$: $v_{ba}=1$, $v_{ac}=v_{bd}=2$, $v_{cb}=3$, (see Fig.\ref{atom}). Here, each of the sets ${\cal F}^k$, $k\in\{1,2,3,4\}$ consists of a single forest: $\{F_k\}={\cal F}^{k}=\tilde{\cal F}^{k}$. We have: $\phi^4=0$, $\phi^3=1$, $\phi^2=3$, $\phi^1=7$ and by definition $\phi^0=\infty$. The convexity inequalities are of the form

\begin{equation*}
\phi^0-\phi^1>\phi^1-\phi^2>\phi^2-\phi^3>\phi^3-\phi^4 ,
\end{equation*}
in which the strict inequality sign holds everywhere. All atoms of all algebras are labeled. The set $\{a,b\}$ is a labeled atom of the algebra $\mathfrak{A}_3$. Moreover, the induced subgraph $F_1|_{\{a,b\}}$ is not a tree, but an empty forest consisting of two empty trees with roots at the vertices $a$ and $b$.

Let's take a closer look at how this happened. Let's follow the trees of type ${\cal T}^\circ_{\tt Y}$. 
Consider the tree $T_1$, which forms the set $\tilde{\cal T}^\circ_{\{ a,b,c\}}$. The set of its vertices ${\tt N}=\{ a,b,c,d\}$ is an element of the trivial algebra $\mathfrak{A}_1$, which consists of ${\tt N}$ and the empty set. If $T_1$ is restricted to the marked atom $\{ a,b,c\}\in\aleph^\bullet_2$ of the previous algebra, then this restriction, in full accordance with the proven theorems, turns out to be a tree consisting of two arcs $(a,c)$ and $(c,b)$. However, this tree is not contained in the set $\tilde{\cal T}^\circ_{\{ a,b\}}$ (see Fig.\ref{atomt}). Because of this circumstance, the restriction of the tree $T_1$ to the atom $\{ a,b\}\in\aleph^\bullet_3$ of the even earlier algebra $\mathfrak{A}_3$ is no longer a tree --- the graph $T_1|_{\{ a,b\}}$ is an empty forest consisting of two empty trees rooted at $a$ and $b$. 

Let us make one more important remark to this example. The arc of minimum weight is not included in the spanning tree of minimum weight.

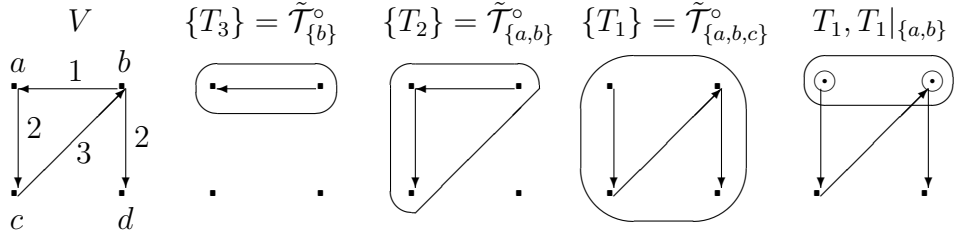
\begin{figure}[h]
\unitlength=1.1mm
\begin{center}
\begin{picture}(115,30)

\put (7,25){$V$}
\put(0,5){$\centerdot$}
\put(13,5){$\centerdot$}
\put(0,18){$\centerdot$}
\put(13,18){$\centerdot$}
\put(0,20){$a$}
\put(13,20){$b$}
\put(0,1){$c$}
\put(13,1){$d$}
\put(1,18){\vector(0,-1){12}}
\put(13,18){\vector(-1,0){12}}
\put(14,18){\vector(0,-1){12}}
\put(1,5){\vector(1,1){13}}
\put(7,19){1}
\put(2,12){2}
\put(8,9){3}
\put(15,11){2}

\put(21,25){$\{ T_3\}=\tilde{\cal T}^\circ_{\{ b\}} $}
\put(24,5){$\centerdot$}
\put(37,5){$\centerdot$}
\put(24,18){$\centerdot$}
\put(37,18){$\centerdot$}
\put(37,18){\vector(-1,0){12}}
\put(31,18){\oval(17,6)}

\put (45,25){$\{ T_2\}=\tilde{\cal T}^\circ_{\{ a,b\}} $}
\put(48,5){$\centerdot$}
\put(61,5){$\centerdot$}
\put(48,18){$\centerdot$}
\put(61,18){$\centerdot$}
\put(61,18){\vector(-1,0){12}}
\put(49,18){\vector(0,-1){12}}
\put(49,18){\oval(6,6)[tl]}
\put(46,18){\line(0,-1){12}}
\put(49,6){\oval(6,6)[bl]}
\put(49,21){\line(1,0){12}}
\put(61,18){\oval(6,6)[tr]}
\put(49,3){\line(1,1){15}}

\put(69,25){$\{ T_1\}=\tilde{\cal T}^\circ_{\{ a,b,c\}} $}
\put(72,5){$\centerdot$}
\put(85,5){$\centerdot$}
\put(72,18){$\centerdot$}
\put(85,18){$\centerdot$}
\put(73,18){\vector(0,-1){12}}
\put(86,18){\vector(0,-1){12}}
\put(73,5){\vector(1,1){13}}
\put(79,12){\oval(20,20)}

\put (97,25){$T_1, T_1|_{\{a,b\}}$}
\put(97,5){$\centerdot$}
\put(110,5){$\centerdot$}
\put(97,18){$\odot $}
\put(110,18){$\odot$}
\put(98,18){\vector(0,-1){12}}
\put(111,18){\vector(0,-1){12}}
\put(98,5){\vector(1,1){13}}
\put(105,19){\oval(18,6)}

\end{picture} 
\caption{\small From left to right: Graph $V$ and its arc weights; trees $T_{3}$, $T_{2}$, $T_{1}$ representing the sets of trees $\tilde{\cal T}^\circ_{\{ b\}}$, $\tilde{\cal T}^\circ_{\{ a,b\}}$, $\tilde{\cal T}^\circ_{\{ a,b,c\}}={\cal T}^\circ_{\{ a,b,c\}}$, respectively; on the right -- tree $T_1$ and its restriction to the labeled atom $\{a,b\}\in\aleph_3^\bullet$. This restriction is not a tree, but an empty forest consisting of two empty trees rooted at $a$ and $b$.}
\label{atomt}
\end{center}
\end{figure}
\end{example}

\subsection{Convexity inequalities contain equality signs}

If in convexity inequalities there are strict inequalities in a row, that is, (\ref{ineq}) is satisfied, then nothing additional to the proven Theorems can be said. Now let both the inequality sign and the equality sign occur.    
\begin{equation}
\phi^{m-1}-\phi^m >\phi^m- \phi^{m+1}=\ldots = 
 \phi^{k-1}-\phi^k >\phi^k- \phi^{k+1} , k-m>1.
 \label{nequal}
\end{equation}
When $k-m=1$ there are no equal signs and we get into the situation (\ref{ineq}).

\begin{theorem}
Let (\ref{nequal}) hold. Then $\mathfrak{A}_{m+1}=\mathfrak{A}_{m+2}=\ldots = \mathfrak{A}_{k}$, and  $\aleph^\bullet_{m+1}=\aleph^\bullet_{m+2}=\ldots = \aleph^\bullet_{k}$,  
\begin{equation}
\tilde{\cal F}^m|_{\uparrow{\tt Y}}\subseteq \tilde{\cal F}^{m+1}|_{\uparrow{\tt Y}}=\tilde{\cal F}^{m+2}|_{\uparrow{\tt Y}}=\cdots =\tilde{\cal F}^{k-1}|_{\uparrow{\tt Y}}\supseteq \tilde{\cal F}^k|_{\uparrow{\tt Y}} \ , \ \ {\tt Y}\in\aleph_k \ . 
\label{calf}
\end{equation}
\end{theorem} 

\begin{proof}
To prove this, it is sufficient to consistently use Property 6 for the equal signs in (\ref{nequal}).
\end{proof}

By Theorem 5, we exclude from consideration the case when there exists an atom ${\tt Z}\in\aleph^\bullet_k$ such that for all $F\in\tilde{\cal F}^{k-1}$ we have ${\tt N}^{out}_{\tt Z}(F)\neq \emptyset$. First, consider the inverse situation when there exists an ${\tt Z}\in\aleph^\bullet_k$ such that ${\tt N}^{out}_{\tt Z}(\tilde{\cal F}^{k-1})=\emptyset$.

\begin{theorem}
Let (\ref{nequal}) hold and let ${\tt Z}\in\aleph^\bullet_k$ be such that ${\tt N}^{out}_{\tt Z}(\tilde{\cal F}^{k-1})=\emptyset$. Then  $\tilde{\cal F}^{l}|_{\uparrow{\tt Z}}= \tilde{\cal F}^{l}|_{\tt Z}=\tilde{\cal T}^\bullet_{\tt Z}$ 
 under $l=m,m+1,\ldots,k$.  
\end{theorem}

\begin{proof}

By Theorem 6, $\tilde{\cal F}^{k-1}|_{\uparrow{\tt Z}} =\tilde{\cal F}^{k-1}|_{\tt Z} =\tilde{\cal T}^\bullet_{\tt Z}$, and by Theorem 1, $\tilde{\cal F}^k|_{\tt Z} =\tilde{\cal T}^\bullet_{\tt Z}$. From (\ref{calf}) we now conclude that

\begin{equation*}
\tilde{\cal F}^m|_{\uparrow{\tt Z}}\subseteq \tilde{\cal F}^{m+1}|_{\uparrow{\tt Z}}= \tilde{\cal F}^{m+2}|_{\uparrow{\tt Z}}=\cdots  =\tilde{\cal F}^{k-1}|_{\uparrow{\tt Z}}= \tilde{\cal F}^k|_{\uparrow{\tt Z}}=\tilde{\cal T}^\bullet_{\tt Z}  . 
\end{equation*}
In particular, ${\tt N}_{\tt Z}^{out}(\tilde{\cal F}^m)=\emptyset$ and $\Upsilon^F_{\tt Z}=\lambda_{\tt Z}^\bullet$. Now let $F\in\tilde{\cal F}^m$ and $T\in\tilde{\cal T}^\bullet_{\tt Z}$. By Property 3, the graph $Q=F^T_{\uparrow{\tt Z}}$ is a forest. Obviously, $Q\in{\cal F}^m$. Let us verify that it is minimal. Indeed,
\begin{equation*}
0\le \Upsilon^Q-\Upsilon^F=\Upsilon^Q_{\tt Z}-\Upsilon^F_{\tt Z} =\lambda_{\tt Z}^\bullet-\Upsilon^F_{\tt Z}= 0 \ .
\end{equation*} 
Thus, $Q\in\tilde{\cal F}^m$, and by construction $Q|_{\uparrow{\tt Z}}=Q|_{\tt Z}=T$. 

\end{proof}

It remains to consider the case when for ${\tt Z}\in\aleph^\bullet_k$ there is a forest $F\in\tilde{\cal F}^{k-1}$ such that ${\tt N}^{out}_{\tt Z}(F)\neq \emptyset$, and there is a forest $G\in\tilde{\cal F}^{k-1}$ such that ${\tt N}^{out}_{\tt Z}(G)= \emptyset$.

\begin{theorem}
Let (\ref{nequal}) hold, and let there exist ${\tt Z}\in\aleph^\bullet_k$, and forests $F,G\in\tilde{\cal F}^{k-1}$ such that ${\tt N}^{out}_{\tt Z}(F)\neq \emptyset$ and ${\tt N}^{out}_{\tt Z}(G)= \emptyset$. Then for $l=m+1,\ldots,k-1$ 

1) there exists a forest $H\in \tilde{\cal F}^{l}$ such that ${\tt N}^{out}_{\tt Z}(H)\neq \emptyset$. For it, $H|_{\uparrow{\tt Z}}\in\tilde{\cal T}^\circ_{\tt Z}$ holds. For any $T\in\tilde{\cal T}^\circ_{\tt Z}$, there exists a forest $Q\in\tilde{\cal F}^{l}$ such that $Q|_{\uparrow{\tt Z}}=T$;  

2) there exists a forest $S\in \tilde{\cal F}^{l}$ such that ${\tt N}^{out}_{\tt Z} (S)= \emptyset$. For it, $S|_{\tt Z}\in\tilde{\cal T}^\bullet_{\tt Z}$. 
For any $T\in\tilde{\cal T}^\bullet_{\tt Z}$, there exists a forest $Q\in\tilde{\cal F}^{l}$ such that $Q|_{\uparrow{\tt Z}} =Q|_{\tt Z}=T$.
\end{theorem}

\begin{proof}
For $l=k-1$ the statement of the theorem is satisfied, since it coincides with Theorem 7. Now for $l=m+1,m+2,\ldots ,k-2$ the statement of the theorem follows from the chain of equalities in (\ref{calf}).
\end{proof}

\begin{theorem}
Let (\ref{nequal}) hold, and let there be ${\tt Z}\in\aleph^\bullet_k$, and forests $F,G\in\tilde{\cal F}^{k-1}$ such that ${\tt N}^{out}_{\tt Z}(F)\neq \emptyset$ and ${\tt N}^{out}_{\tt Z}(G)= \emptyset$. 
Then ${\tt N}^{out}_{\tt Z}(\tilde{\cal F}^m)\neq \emptyset$ and

1) there exists a forest $H\in\tilde{\cal F}^m$ such that ${\tt N}^{out}_{\tt Z}(H)\neq \emptyset$. For it, $H|_{\tt Z}\in \tilde{\cal T}_{\tt Z}^\circ$. In particular, if for all $H\in\tilde{\cal F}^m$, ${\tt N}^{out}_{\tt Z}(H)\neq \emptyset$, then $\tilde{\cal F}^m|_{\uparrow{\tt Z}}\subseteq\tilde{\cal T}^\circ_{\tt Z}$.

2) if there exists a forest $S\in\tilde{\cal F}^m$ such that ${\tt N}^{out}_{\tt Z}(S)= \emptyset$, then $S|_{\tt Z}\in\tilde{\cal T}^\bullet_{\tt Z}$ holds for it. In this case, for any tree $T\in\tilde{\cal T}^\bullet_{\tt Z}$, there exists a forest $Q\in\tilde{\cal F}^m$ such that $Q|_{\uparrow{\tt Z}} =Q|_{\tt Z}=T$.  

\end{theorem}

\begin{figure}[h]
\unitlength=0.9mm
\begin{center}
\begin{picture}(147,35)

\put(0,32){$V$}
\put(4,5){\vector(1,0){17}}
\put(4,7){\vector(0,1){17}}
\put(21,7){\vector(-1,0){17}}
\put(21,25){\vector(-1,0){16}}
\put(5,15){\small 2}
\put(12,26){\small 1}
\put(12,1){\small 2}
\put(12,8){\small 2}
\put(1,5){$\odot$}
\put(21,5){$\odot$}
\put(1,24){$\odot$}
\put(21,24){$\odot$}
\put(5,9){$\tt Z$}
\put(20,9){$\tt L$}
\put(3,27){$\tt J$}
\put(20,27){$\tt I$}

\put(30,32){$H$}
\put(34,7){\vector(0,1){17}}
\put(51,25){\vector(-1,0){16}}
\put(51,7){\vector(-1,0){17}}
\put(35,15){\small 2}
\put(42,26){\small 1}
\put(42,8){\small 2}
\put(31,5){$\odot$}
\put(51,5){$\odot$}
\put(31,24){$\odot$}
\put(51,24){$\odot$}
\put(35,9){$\tt Z$}
\put(50,9){$\tt L$}
\put(33,27){$\tt J$}
\put(50,27){$\tt I$}

\put(60,32){$F$}
\put(81,25){\vector(-1,0){16}}
\put(64,7){\vector(0,1){17}}
\put(65,15){\small 2}
\put(72,26){\small 1}
\put(61,5){$\odot$}
\put(81,5){$\odot$}
\put(61,24){$\odot$}
\put(81,24){$\odot$}
\put(65,9){$\tt Z$}
\put(80,9){$\tt L$}
\put(63,27){$\tt J$}
\put(80,27){$\tt I$}
\put(43,26){\oval(25,13)}
\put(73,26){\oval(25,13)}

\put(90,32){$G$}
\put(111,6){\vector(-1,0){16}}
\put(111,25){\vector(-1,0){16}}
\put(102,8){\small 2}
\put(102,26){\small 1}
\put(91,5){$\odot$}
\put(111,5){$\odot$}
\put(91,24){$\odot$}
\put(111,24){$\odot$}
\put(95,9){$\tt Z$}
\put(110,9){$\tt L$}
\put(93,27){$\tt J$} 
\put(110,27){$\tt I$}
\put(103,26){\oval(25,13)}

\put(120,32){$G'$}
\put(131,7){$T$}
\put(125,6){\vector(1,0){16}}
\put(141,25){\vector(-1,0){16}}
\put(132,2){\small 2}
\put(132,26){\small 1}
\put(121,5){$\odot$}
\put(141,5){$\odot$}
\put(121,24){$\odot$}
\put(141,24){$\odot$}
\put(124,9){$\tt Z$}
\put(140,9){$\tt L$}
\put(123,28){$\tt J$}
\put(140,27){$\tt I$}
\put(133,7){\oval(25,13)}
\put(133,26){\oval(25,13)}

\end{picture} 
\caption{\small Depicted from left to right: graph $V$; forest $H$ constituting the set $\tilde{\cal F}^1={\cal F}^1$; forests $F$, $G$, and $G'$ constituting the set ${\tilde{\cal F}^2}$. Here $\phi^4=0$, $\phi^3=1$, $\phi^2=3$, $\phi^1=5$, and $\phi^0=\infty$ by definition. We have $\phi^0-\phi^1> \phi^1-\phi^2= \phi^2-\phi^3 >\phi^3-\phi^4$ ($m=1, \ k=3$). The labeled atoms that make up $\aleph^\bullet_3=\aleph^\bullet_2$ are the vertices ${\tt Z}$, ${\tt L}$ of the graph $V$ and the set $\{ \tt I,J \}$. 
The conditions of Theorem 11 are satisfied: ${\tt N}^{out}_{\tt Z}(F)\neq \emptyset$ and ${\tt N}^{out}_{\tt Z}(G)= \emptyset$. 
 Since $H$ is the only forest that makes up the set $\tilde{\cal F}^1$, and in it ${\tt N}^{out}_{\tt Z}(H)=\{ \tt J\}\neq \emptyset$, there is no forest $S\in\tilde{\cal F}^1$ such that ${\tt N}^{out}_{\tt Z}(S)=\emptyset$. 
In the forest $G'$, the tree $T\in\tilde{\cal T}_{\tt Z}^\circ$ (the tree with the only arc $(\tt Z,L)$) is additionally depicted. The set of trees $\tilde{\cal F}^1|_{\uparrow{\tt Z}}=\{ H\}$ does not contain this tree.}
\label{p5=}
\end{center}
\end{figure}
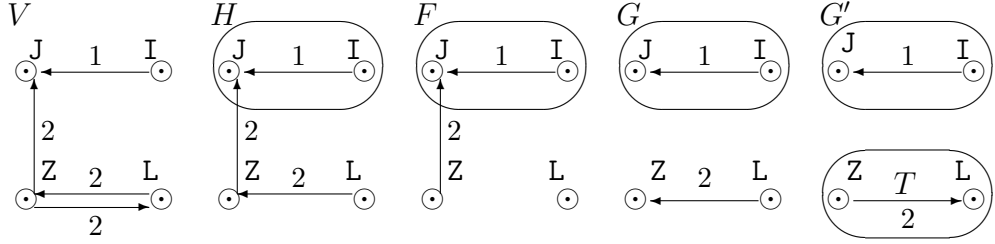

\begin{proof} 

Let us first verify that the case when ${\tt N}^{out}_{\tt Z}(H)= \emptyset$ holds for all $H\in\tilde{\cal F}^m$ is impossible. Indeed, by the condition, there exists $F\in\tilde{\cal F}^{k-1}$ such that ${\tt N}^{out}_{\tt Z}(F)\neq \emptyset$. By Property 4, it has minimal descendants and great-descendants $R\in\tilde{\cal F}^l$ for any $1\leq l<k-1$ (we assume that ${\cal F}^1\neq \emptyset$, although here it is sufficient that ${\cal F}^m\neq \emptyset$). And for all of them, ${\tt N}^{out}_{\tt Z}(R)\neq \emptyset$ holds. Therefore, for $R\in\tilde{\cal F}^m$, also ${\tt N}^{out}_{\tt Z}(R)\neq \emptyset$.

1) The fact that there exists $H\in\tilde{\cal F}^m$, for which ${\tt N}^{out}_{\tt Z}(H)\neq \emptyset$ is already clear to us. By Property 4, any forest in $\tilde{\cal F}^m$ has a relative in $\tilde{\cal F}^{m+1}$ and vice versa. From Theorem 10 we know that for any $P\in\tilde{\cal F}^{m+1} $, if 
${\tt N}^{out}_{\tt Z}(P)\neq \emptyset$, then $P|_{\uparrow{\tt Z}}\in \tilde{\cal T}^\circ_{\tt Z}$. 
By Theorem 8 the left inclusion (\ref{calf}) holds

\begin{equation}
\tilde{\cal F}^m|_{\uparrow{\tt Z}}\subseteq \tilde{\cal F}^{m+1}|_{\uparrow{\tt Z}} \ . 
\label{calf1}
\end{equation}
Therefore, for the forest $H$, $H|_{\uparrow{\tt Z}}\in \tilde{\cal T}^\circ_{\tt Z}$ holds. If, however, for 
for all $H\in\tilde{\cal F}^m$, ${\tt N}^{out}_{\tt Z}(H)\neq \emptyset$ holds, then this means that $\tilde{\cal F}^m|_{\uparrow{\tt Z}}\subseteq\tilde{\cal T}^\circ_{\tt Z}$. 

2) It follows from Theorem 10 and (\ref{calf1}) that if there exists a 
forest $S\in\tilde{\cal F}^m$ such that ${\tt N}^{out}_{\tt Z}(S)= \emptyset$, then for it, $S|_{\uparrow{\tt Z}}= S|_{\tt Z}\in\tilde{\cal T}^\bullet_{\tt Z}$. 

Now let $T\in\tilde{\cal T}^\bullet_{\tt Z}$ and let $S\in \tilde{\cal F}^m$ be such that ${\tt N}^{out}_{\tt Z}(S)= \emptyset$ (see Fig.\ref{p52}). Since ${\tt N}^{out}_{\tt Z}(T)=\emptyset$, by Property 3 the graph $Q={S}_{\uparrow{\tt Z}}^{T}$ is a forest. In it $\Upsilon^Q_{\overline{\tt Z}}=\Upsilon^S_{\overline{\tt Z}}$ and $\Upsilon^Q_{\tt Z}=\Upsilon^T= \lambda^\bullet_{\tt Z}$. Obviously, $Q\in{\cal F}^m$ and hence 

\begin{equation*}
0\leq \Upsilon^Q - \Upsilon^{S} =\Upsilon^Q_{\overline{\tt Z}}+\Upsilon^Q_{\tt Z} -(\Upsilon^{S}_{\overline{\tt Z}}+ \Upsilon^{S}_{\tt Z}) = \Upsilon^Q_{\tt Z}-\Upsilon^S_{\tt Z}= \lambda^\bullet_{\tt Z} -\lambda^\bullet_{\tt Z}  =0 \ .  
\end{equation*}
From this we conclude that $Q\in \tilde{\cal F}^m$, and by construction $Q|_{\uparrow{\tt Z}} =Q|_{\tt Z}=T$. 
\end{proof}

\begin{figure}[h]
\unitlength=0.75mm
\begin{center}
\begin{picture}(177,32)

\put(12,29){$V$}
\put(4,5){\vector(1,0){17}}
\put(4,8){\vector(0,1){15}}
\put(21,7){\vector(-1,0){16}}
\put(23,23){\vector(0,-1){15}}
\put(3,23){\vector(0,-1){15}}
\put(5,15){\small 2}
\put(0,15){\small 2}
\put(20,15){\small 2}
\put(12,0){\small 1}
\put(12,8){\small 1}
\put(1,5){$\odot$}
\put(21,5){$\odot$}
\put(1,24){$\odot$}
\put(21,24){$\odot$}
\put(5,22){$\tt J$}
\put(18,22){$\tt I$}
\put(20,0){$\tt L$}
\put(4,0){$\tt L'$}

\put(42,29){$H$}
\put(33,8){\vector(0,1){15}}
\put(53,23){\vector(0,-1){15}}
\put(51,6){\vector(-1,0){15}}
\put(34,15){\small 2}
\put(50,15){\small 2}
\put(42,1){\small 1}
\put(31,5){$\odot$}
\put(51,5){$\odot$}
\put(31,24){$\odot$}
\put(51,24){$\odot$}
\put(50,0){$\tt L$}
\put(34,0){$\tt L'$}
\put(35,22){$\tt J$}
\put(48,22){$\tt I$}
\put(42,8){\tt Z}
\put(43,6){\oval(26,14)}

\put(72,29){$Q$}
\put(83,23){\vector(0,-1){15}}
\put(63,23){\vector(0,-1){15}}
\put(65,6){\vector(1,0){16}}
\put(60,15){\small 2}
\put(72,1){\small 1}
\put(80,15){\small 2}
\put(61,5){$\odot$}
\put(81,5){$\odot$}
\put(61,24){$\odot$}
\put(81,24){$\odot$}
\put(65,22){$\tt J$}
\put(78,22){$\tt I$}
\put(80,0){$\tt L$}
\put(64,0){$\tt L'$}
\put(72,8){$T'$}
\put(73,6){\oval(26,14)}

\put(102,29){$S$}
\put(93,23){\vector(0,-1){15}}
\put(113,23){\vector(0,-1){15}}
\put(111,6){\vector(-1,0){16}}
\put(90,15){\small 2}
\put(110,15){\small 2}
\put(102,1){\small 1}
\put(91,5){$\odot$}
\put(111,5){$\odot$}
\put(91,24){$\odot$}
\put(111,24){$\odot$}
\put(95,22){$\tt J$} 
\put(108,22){$\tt I$}
\put(110,0){$\tt L$}
\put(94,0){$\tt L'$}
\put(102,8){$T$}
\put(103,06){\oval(26,14)}

\put(132,29){$F$}
\put(123,8){\vector(0,1){15}}
\put(141,6){\vector(-1,0){15}}
\put(124,15){\small 2}
\put(132,1){\small 1}
\put(121,5){$\odot$}
\put(141,5){$\odot$}
\put(121,24){$\odot$}
\put(141,24){$\odot$}
\put(140,0){$\tt L$}
\put(124,0){$\tt L'$}
\put(125,22){$\tt J$}
\put(138,22){$\tt I$}
\put(132,8){$T$}
\put(133,6){\oval(26,14)}

\put(162,29){$G$}
\put(173,23){\vector(0,-1){15}}
\put(171,6){\vector(-1,0){15}}
\put(170,15){\small 2}
\put(162,1){\small 1}
\put(151,5){$\odot$}
\put(171,5){$\odot$}
\put(151,24){$\odot$}
\put(171,24){$\odot$}
\put(170,0){$\tt L$}
\put(154,0){$\tt L'$}
\put(155,22){$\tt J$}
\put(168,22){$\tt I$}
\put(162,8){$T'$}
\put(163,6){\oval(26,14)}
\end{picture} 
\caption{\small Depicted from left to right are: graph $V$; forests $H$, $Q$, and $S$, which form the set $\tilde{\cal F}^1={\cal F}^1$; forests $F$ and $G$, which form the set $\tilde{\cal F}^2$. The labeled atoms, which form $\aleph^\bullet_3=\aleph^\bullet_2$, are the vertices ${\tt I}$, ${\tt J}$ of graph $V$ and the set $\tt Z=\{ L,L' \}$. Here $\phi^4=0$, $\phi^3=1$, $\phi^2=3$, $\phi^1=5$, and $\phi^0=\infty$ by definition. We have: $\phi^0-\phi^1>\phi^1-\phi^2=\phi^2-\phi^3> \phi^3-\phi^1$ ($m=1, \ k=3$). The conditions of Theorem 11 are satisfied: ${\tt N}^{out}_{\tt Z}(F)\neq \emptyset$ and  ${\tt N}^{out}_{\tt Z}(G)= \emptyset$, also ${\tt N}^{out}_{\tt Z}(S)= \emptyset$ (as in $Q$). The set of trees $\tilde{\cal T}_{\tt Z}^\bullet$ consists of two trees $T$ and $T'$, each having one single arc: in tree $T$ --- arc $\tt (L,L')$, in tree $T'$ --- $(\tt L',L)$. Forests $S$ and $Q$ contain one of these trees each. In these forests, the arc does not originate from $\tt Z$. But there must be a forest of $\tilde{\cal F}^1$ (a forest of $H$) in which an arc emanates from ${\tt Z}$. The set $\tilde{\cal T}^\circ_{\tt Z}$ consists of one tree in which there are two arcs. These are the arcs $({\tt L},{\tt L'})$ and $({\tt L'},{\tt J})$. Done $\tilde{\cal F}^m|_{\uparrow{\tt Z}}= \{H |_{\uparrow{\tt Z}} \}=\tilde{\cal T}^\circ_{\tt Z}$.}
\label{p52}
\end{center}
\end{figure}
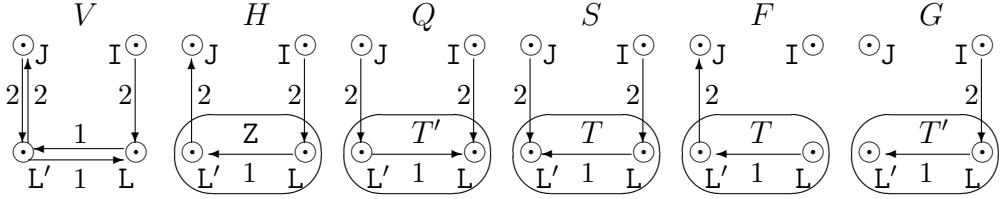

Note that Theorem 11 cannot be strengthened. For point 1), it is indeed possible that for all $H\in\tilde{\cal F}^m$ ${\tt N}^{out}_{\tt Z}(H)\neq \emptyset$ holds (see Fig.\ref{p5=}). Moreover, it cannot be asserted that if for all $H\in\tilde{\cal F}^m$ ${\tt N}^{out}_{\tt Z}(H)\neq \emptyset$ holds, then $\tilde{\cal F}^m|_{\uparrow{\tt Z}} =\tilde{\cal T}^\circ_{\tt Z}$. A situation is possible when there exists $T\in\tilde{\cal T}^\circ_{\tt Z}$ such that $T\notin\tilde{\cal F}^m|_{\uparrow{\tt Z}}$ (see the same figure).

For point 2), a variant is possible when there is no $S\in\tilde{\cal F}^m$ for which ${\tt N}^{out}_{\tt Z}(S)= \emptyset$, which is again shown by the same figure.

Fig.\ref{p52} shows a variant when there is a forest $S$ for which ${\tt N}^{out}_{\tt Z}(S)= \emptyset$. In this case, one of the situations of point 1) occurs, when $\tilde{\cal F}^m|_{\uparrow{\tt Z}}=\tilde{\cal T}^\circ_{\tt Z}$.

Theorems 8–11 allow us to formulate a conclusion similar to Conclusion 1. 
\begin{conclusion}
Let (\ref{nequal}) hold and let our knowledge of the original graph $V$ be limited to the fact that we know the partition $\aleph_k$ and all trees from the sets $\tilde{\cal T}_{\tt Y}^\bullet$ and $\tilde{\cal T}_{\tt Y}^\circ$ for any ${\tt Y}\in\aleph_k$.
Then we can find all forests $\tilde{\cal F}^l$ for $l=m,m+1,\ldots,k$. Thus, in particular, the set of forests $\tilde{\cal F}^m$ is known. And this means that the new algebra $\mathfrak{A}_m$ can be defined (the previous ones from $m+1$ to $k$ coincide), and all its labeled and unlabeled atoms ${\tt Z'}\in\aleph_m^\bullet$ and ${\tt X'}\in\aleph_m^\circ$, and the sets $\tilde{\cal T}_{\tt Z'}^\bullet$ and $\tilde{\cal T}_{\tt X'}^\circ$. But the sets $\tilde{\cal T}_{\tt Z'}^\circ$ are not known. And this does not allow us to build forests of $\tilde{\cal F}^{m-1}$ without using additional information. Our knowledge ends with the fact that by Theorem 3 we know that $ \tilde{\cal F}^{m-1}|_{\uparrow{\tt X}'}\subseteq \tilde{\cal T}^\circ_{\tt X'}$, for ${\tt X'}\in\aleph_m^\circ$. Thus, for $F\in\tilde{\cal F}^{m-1}$ we know not only that $F|_{\uparrow{\tt X'}}$ is a tree, but we also know the tree itself. This means that we know, in particular, that for those ${\tt Y}\in\aleph_k$ that are contained in ${\tt X'}$, $F|_{\uparrow{\tt Y}}$ is a tree, and we also know this tree. 

Furthermore, from Property 15, in particular, it follows that for $F\in\tilde{\cal F}^{m-1}$ the induced subgraph $F|_{\tt Z'}$ is a tree for ${\tt Z'}\in\aleph_m^\bullet$. But we do not know at all what kind of tree this is. And the reason for this is that for those ${\tt Y}\in\aleph_k$ that are contained in ${\tt Z'}$, the graph $F|_{\uparrow{\tt Y}}$ may (and the examples considered show this --- see Fig.\ref{atom},\ref{atomt}) not be a tree. The tree structure of these atoms (atoms of the algebra $\mathfrak{A}_k$) is destroyed in forests of $\tilde{\cal F}^{m-1}$. 
Therefore, even if we supplement our knowledge with all the arcs of the original graph whose ends belong to different atoms of the algebra $\mathfrak{A}_k$, this will not help us in constructing forests from $\tilde{\cal F}^{m-1}$. For such construction, we need to additionally know not the above-mentioned arcs, but all the sets of trees $\tilde{\cal T}_{\tt Z'}^\circ$ for ${\tt Z'}\in\aleph_m^\bullet$. And then you can use all the proven theorems, but starting in descending order not from index $k$, but from index $m$.  

\end{conclusion}
\section{About Markov chains} 

The intensity matrix $\bf A$ of a continuous-time Markov chain (the infinitesimal operator of the semigroup of transition probability matrices ${\bf P}(t)$) is, up to a sign, the Laplace matrix ${\bf L}=-{\bf A}$ \cite{V3} of the weighted directed graph $A$ without loops corresponding to this chain. 
In the presence of process killing, loops are present and the Laplace matrix is undefined, but the loops can be eliminated by introducing an additional killing state, which allows the Laplace matrix to be introduced in this case as well  \cite{V4}. 

The coefficients at the degrees $\lambda^l$ of the characteristic polynomial and the coefficients at the same degrees $\lambda^l$ of the components of the corresponding eigenvectors of the Laplacian represent an unsigned sum over spanning forests ${\cal F}^l(A)$ \cite{V3}-\cite{V5}.  

Let the transition densities be exponentially small \cite{VF} $a_{ij}\thicksim \exp(-v_{ij}/\varepsilon)$. This is a general situation in low-temperature physics \cite{VP}. The unsignedness of the sums at powers of the spectral parameter $\lambda$ allows us to follow in the asymptotics only the terms of maximum order \cite{V1,V2,V}. The maximum order is reached when the sum of the corresponding values of $v_{ij}$ is minimal. This corresponds to the sum of products of matrix elements over minimal forests $F\in\tilde{\cal F}^l$ of graph $V$. And the arcs of forests from $\tilde{\cal F}^l$ can be grouped according to whether their outcomes belong to the atoms of the algebra $\mathfrak{A}_l$. This corresponds to trees from the sets $\tilde{\cal T}^{\bullet}_{\tt Y}$ and $\tilde{\cal T}^{\circ}_{\tt Y}$ for ${\tt Y}\in\aleph_l$.   



\centerline{Abstract}

\begin{center}{How Trees on Atoms of Subset Algebras Define Minimal Forests and Their Growth}
\end{center}

\centerline{Buslov V.A.}

\parbox[t]{12cm}
{\small A complete description is given of how minimal trees on atoms of the algebra of subsets $\mathfrak{A}_k$ generated by minimal spanning $k$-component forests of a weighted digraph $V$ determine the form of these forests and how forests grow with increasing number of arcs (that is with a decrease in the number of trees). Precise bounds are established on what can be extracted about the tree structure of the original graph if the minimal trees on the atoms of a single algebra $\mathfrak{A}_k$ are known, and also what minimum spanning forests with fewer components can be constructed based on this, and what exactly additional information is required to determine minimum spanning forests consisting of even fewer components.
}
\vspace{0.5cm}

St. Petersburg State University, Faculty of Physics, Department of Computational Physics

198504 St. Petersburg, Old Peterhof, st. Ulyanovskaya, 3

Email: abvabv@bk.ru, v.buslov@spbu.ru
\end{document}